\documentclass[11pt,a4paper]{article}

\usepackage{amssymb, amsmath, amsthm}
\usepackage{graphicx}
\newcommand{\E}{\mathbb{E}}
\newcommand{\p}{\mathbb{P}}
\newcommand{\R}{\mathbb{R}}
\newcommand{\dd}{\mathrm{d}}
\newcommand{\var}{\mathbf{Var}}

\newtheorem{theorem}{Theorem}[section]
\newtheorem{proposition}[theorem]{Proposition}
\newtheorem{lemma}[theorem]{Lemma}

\theoremstyle{remark}
\newtheorem*{example}{Example}

\author{}

\date{}

\begin{document}


\begin{center}
{\LARGE Limit laws for the norms of extremal samples}

\bigskip
\bigskip

\textsc{P\'eter Kevei}\footnote{
kevei@math.u-szeged.hu},  
\textsc{Lillian Oluoch}\footnote{
oluoch@math.u-szeged.hu}, and 
\textsc{L\'aszl\'o Viharos}\footnote{
viharos@math.u-szeged.hu} \\
\medskip
Bolyai Institute, University of Szeged \\ Aradi v\'ertan\'uk 
tere 1, 6720, Szeged, Hungary
\end{center}

\bigskip

\begin{abstract}
Let denote
$S_n(p) = k_n^{-1} \sum_{i=1}^{k_n} 
\left( \log (X_{n+1-i,n} / X_{n-k_n, n}) \right)^p$,
where $p > 0$, $k_n \leq n$ is a sequence of integers such that
$k_n \to \infty$ and $k_n / n \to 0$, and 
$X_{1,n} \leq \ldots \leq X_{n,n}$ is the order statistics
of iid random variables with regularly varying upper tail.
The estimator $\widehat \gamma(n) = (S_n(p)/\Gamma(p+1))^{1/p}$
is an extension of the Hill estimator.
We investigate the asymptotic properties of $S_n(p)$ 
and $\widehat \gamma(n)$ both
for fixed $p > 0$ and for $p = p_n \to \infty$. We prove 
strong consistency and asymptotic normality under 
appropriate assumptions. Applied to real data we find that
for larger $p$ the estimator is less sensitive to the change
in $k_n$ than the Hill estimator.

\end{abstract}

\section{Introduction} \label{sect:intro}

Let $X, X_1, X_2,\ldots$ be iid random variables with common distribution
function $F(x) = \p ( X \leq x)$, $x \in \R$. For each $n \geq 1$, let 
$X_{1,n} \leq \ldots \leq X_{n,n}$ denote the order statistics of the 
sample $X_1, \ldots, X_n$. Assume that 
\[
1  - F(x) = x^{-1/\gamma} L(x),
\] 
where $L$ is a slowly varying function at infinity and $\gamma > 0$. 
This is equivalent to the condition 
\begin{equation} \label{eq:quant}
Q( 1- s) = s^{-\gamma} \ell(s),
\end{equation}
where $Q(s) = \inf \{ x: \, F(x) \geq s \}$, $s \in (0,1)$,
stands for the quantile function, and 
$\ell$ is a slowly varying function at 0. 
For $p > 0$ introduce the notation
\begin{equation} \label{eq:def-sn}
S_n(p) = \frac{1}{k_n} \sum_{i=1}^{k_n} 
\left( \log \frac{X_{n+1-i,n}}{X_{n-k_n,n}} \right)^p.
\end{equation}
In what follows we always assume that $1 \leq k_n \leq n$ is 
a sequence of integers such that 
$k_n \to \infty$ and $k_n / n \to 0$.

\smallskip

As a special case for $p=1$ we obtain the well-known Hill
estimator of the tail index $\gamma> 0$ introduced 
by Hill in 1975 \cite{Hill}. 
For $p = 2$ the estimator was suggested by 
Dekkers et al.~\cite{Dekkers}, where they proved that 
$S_n(2) \to 2 \gamma^2$ a.s.~or in probability, 
depending on the assumptions on $k_n$, and they proved asymptotic
normality of the estimator as well. Segers \cite{Segers2} 
considered more general estimators of the form
\[
\frac{1}{k_n} \sum_{i=1}^{k_n} 
f\left( \frac{X_{n+1-i,n}}{X_{n-k_n,n}} \right),
\]
for a nice class of functions $f$, called residual estimators.
Segers proved weak consistency and asymptotic normality under 
general conditions. More recently, Ciuperca and Mercadier
\cite{Ciuperca} investigated weighted version of 
\eqref{eq:def-sn} and obtained weak consistency and 
asymptotic normality for the estimator. 

To the best of our 
knowledge the possibility $p = p_n \to \infty$ was not 
considered before. The estimate of the tail index 
\begin{equation*} 
\widehat \gamma(n) =
\left( \frac{S_n(p_n)}{\Gamma(p_n+1)} 
\right)^{\frac{1}{p_n}}
\end{equation*} 
can be considered as $p_n \to \infty$
as the limit law for the norm of the extremal sample.
In this direction Schlather \cite{Schlather} and 
Bogachev \cite{Bogachev} proved limit theorems for 
norms of iid samples.

\smallskip

In the present paper we investigate the asymptotic properties 
of $S_n(p)$ and $\widehat \gamma(n)$ 
both for $p > 0$ fixed and for $p = p_n \to \infty$.
In Sections \ref{sect:cons} and \ref{sect:normality} $p$ is fixed,
while it tends to infinity in Section \ref{sect:infty}.
In Theorem \ref{thm:strong-cons} we prove strong consistency of
the estimator for fixed $p$. Strong consistency was only 
obtained by Dekkers et al.~\cite{Dekkers} for $p=1$ and $p=2$,
thus our result is new for general $p$. Asymptotic normality 
is treated in Section \ref{sect:normality}. In this direction
very general results was obtained by Segers \cite{Segers2} for
residual estimators. However, our assumptions 
in Theorem \ref{thm:est-norm} on the slowly
varying function $\ell$ are weaker than in Theorem 4.5 in
\cite{Segers2}. In Section \ref{sect:infty} we obtain 
weak consistency and asymptotic normality when $p \to \infty$.
Section \ref{sect:sim} contains the simulation results
and data analysis. Here we show that for larger values of $p$
the estimator is not so sensitive to the choice of $k_n$, which
is a critical property in applications. We demonstrate this 
property on the well-known dataset of Danish fire insurance
claims, see Resnick \cite{Resn} and 
Embrechts et al.~\cite[Example 6.2.9]{EKM}.
The technical proofs are gathered together in Section \ref{sect:proofs}.

\section{Consistency} \label{sect:cons}

In what follows, $U, U_1, U_2, \ldots$ are iid uniform$(0,1)$
random variables, and $U_{1,n} \leq U_{2,n} \leq \ldots \leq U_{n,n}$
stands for the order statistics.
To ease notation we frequently suppress the dependence on $n$ and
simply write $k=k_n$.
According to the well-known quantile representation, we have
\[ 
\begin{split}
& (X_{1,n}, X_{2,n}, \ldots, X_{n,n} )_{n \geq 1} 
 \stackrel{\mathcal{D}}{=}
(Q(U_{1,n}), Q(U_{2,n}), \ldots, Q(U_{n,n}) )_{n \geq 1} \\
& \stackrel{\mathcal{D}}{=}
(Q(1 -U_{n,n}), Q(1-U_{n-1,n}), \ldots, Q(1-U_{1,n}) )_{n \geq 1},
\end{split}
\] 
which implies that $S_n$ in \eqref{eq:def-sn} can be written as
\begin{equation} \label{eq:Sn-repr}
S_n(p) = \frac{1}{k} \sum_{i=1}^k 
\left( \log \frac{Q(1-U_{i,n})}{Q(1-U_{k+1,n})} 
\right)^p \quad \text{for each } n \geq 1, \ \text{a.s.}
\end{equation}
In what follows we use this representation.
Therefore, to understand the behavior of $S_n(p)$ first we have 
to handle uniform random variables.
In the following $\Gamma(x) = \int_0^\infty y^{x-1} e^{-y} \dd y$,
$x > 0$, stands for the usual gamma function.

\begin{lemma} \label{lemma:stoch}
For any sequence $(k_n)$ such that 
$k_n \to \infty$ and $k_n \leq n$, we have
\[
\frac{1}{k_n} \sum_{i=1}^{k_n} 
\left( - \log \frac{U_{i,n}}{U_{k_n+1,n}} \right)^p 
\stackrel{\p}{\longrightarrow} \Gamma(p+1).
\]
\end{lemma}

\begin{proof}
One only has to notice that the sequence 
$(U_{i,n}/U_{k+1,n})_{i=1,\ldots,k}$ has the distribution as 
$(\widetilde U_{i,k})_{i=1,\ldots,k}$, where 
$\widetilde U_1, \widetilde U_2, \ldots$ are iid uniform$(0,1)$
random variables. 
Noting that $\E (- \log U)^p = \Gamma(p+1)$, 
the statement follows from the law of large numbers.
\end{proof}

We note that the representation above immediately implies 
the asymptotic normality
\[
\frac{1}{\sqrt{k_n} \sigma_{p,1}} \sum_{i=1}^{k_n} 
\left[ \left( - \log \frac{U_{i,n}}{U_{k_n+1,n}} \right)^p
- \Gamma(p+1) \right] \stackrel{\mathcal{D}}{\longrightarrow}
\mathrm{N}(0,1),
\]
with $\sigma_{p,1}^2 = \var (( - \log U)^p)$.

For the almost sure version we need some assumption on $k_n$.

\begin{lemma} \label{lemma:as}
Assume that  $k_n/ (\log n)^\delta \to \infty$
for some $\delta > 0$, and $k_n / n \to 0$. Then
\[
\frac{1}{k_n} \sum_{i=1}^{k_n} 
\left( - \log \frac{U_{i,n}}{U_{k_n+1,n}} \right)^p 
\longrightarrow \Gamma(p+1) \quad \text{a.s.}
\] 
\end{lemma}

First we show strong consistency for $S_n(p)$. Our assumption on 
the sequence $k_n$ is the same as in Theorem 2.1 in \cite{Dekkers}.
This is not far from the optimal condition $k_n / \log \log n \to \infty$,
which was obtained by Deheuvels et al.~\cite{DHM88}.

\begin{theorem} \label{thm:strong-cons}
Assume that \eqref{eq:quant} holds and 
$k_n / n \to 0$, $(\log n)^\delta / k_n \to 0$
for some $\delta > 0$. Then
$S_n(p) \to \gamma^p \Gamma( p +1)$ a.s., that is 
for $p > 0$ fixed the estimator $\widehat \gamma(n)$ is 
strongly consistent.
\end{theorem}

Weak consistency holds under weaker assumption on $k_n$.
The following result is a special case of Theorem 2.1 
in \cite{Segers2}, and it follows from representation
\eqref{eq:Sn-repr} and from the law of large numbers.

\begin{theorem} \label{thm:weak-cons}
Assume that \eqref{eq:quant} holds, and the sequence 
$(k_n)$ is such that $k_n \to \infty$, $k_n / n \to 0$.
Then
$S_n(p) \stackrel{\p}{\longrightarrow} \gamma^{p} \Gamma(p+1)$,
that is for $p > 0$ fixed the estimator $\widehat\gamma(n)$
is weakly consistent. 
\end{theorem}

\section{Asymptotic normality} \label{sect:normality}

To prove asymptotic normality we use that in
representation \eqref{eq:Sn-repr} the summands are
independent and identically distributed conditioned on
$U_{k+1,n}$.
Indeed, conditioned on $U_{k+1,n}$
\begin{equation} \label{eq:unif-repr}
(U_{1,n}, \ldots , U_{k,n})
\stackrel{\mathcal{D}}{=}
\left( \widetilde U_{1,k} U_{k+1,n}, \ldots ,
\widetilde U_{k,k} U_{k+1,n} \right),
\end{equation}
where $\widetilde U_1, \widetilde U_2, \ldots$
are iid uniform$(0,1)$ random variables, independent
of $U_{k+1,n}$, and 
$\widetilde U_{1,k} < \ldots < \widetilde U_{k,k}$
stands for the order statistics of $\widetilde U_1$,
$\ldots$, $\widetilde U_k$.

To state the result, we need some notation.
Introduce the variable for $v \in [0,1)$
\begin{equation} \label{eq:defY}
Y(v) = \log \frac{Q(1- Uv)}{Q(1-v)} ,  
\end{equation}
where $U$ is uniform$(0,1)$, and 
$Y(0) = - \gamma \log U$. Define
\[ 
m_{p, \gamma} (v) = m_p(v) =  \E Y(v)^p, \quad 
\sigma_{p, \gamma}^2(v) = \sigma_p^2(v) =  \var Y(v)^p,
\]
and the corresponding limiting quantities 
\[
\begin{split}
& m_p = m_{p, \gamma} = \E (- \gamma \log U)^p = \gamma^p \Gamma(p+1), \\
& \sigma_p^2 = \sigma_{p, \gamma}^2 = \var ((-\gamma \log U)^p ) = 
\gamma^{2p} \left( \Gamma(2p+1) - \Gamma(p+1)^2 \right).
\end{split}
\]
Note that the quantities $m_p$, $\sigma_p$, $m_p(v)$, $\sigma_p(v)$ 
depend on the parameter $\gamma$. However, since the value $\gamma > 0$
is fixed, to ease notation we suppress $\gamma$ in the following.

Central limit theorem with
random centering was obtained in Theorem 4.1 in \cite{Segers2}. 
Next, we spell out this result in our case.
In the special case $p=1$ we obtain 
Theorem 1.6 by Cs\"org\H{o} and Mason \cite{CsM85}. The key
observation in the proof is the representation \eqref{eq:unif-repr}.

\begin{theorem} \label{thm:asnorm}
Assume that \eqref{eq:quant} holds, and $k_n \to \infty$, $k_n/n \to 0$.
Then as $n \to \infty$
\[
\frac{1}{\sqrt{k_n}} \sum_{i=1}^{k_n}
\left[ 
\left( \log \frac{Q(1-U_{i,n})}{Q(1- U_{k_n+1,n})} \right)^p 
- m_p(U_{k+1,n}) \right] 
\stackrel{\mathcal{D}}{\longrightarrow} N(0, \sigma_p^2),
\]
with $\sigma_p^2 = \gamma^{2p} (\Gamma(2p+1) - \Gamma(p+1)^2)$.
\end{theorem}

To obtain asymptotic normality for the estimator,
i.e.~to change the random centering $m_p(U_{k+1,n})$ to $m_p$,
we have to show that
\[
\sqrt{k_n } ( m_p(U_{k+1,n}) - m_p) \stackrel{\p}{\longrightarrow} 0.
\]
Since $U_{k+1,n} n/k \to 1$ in probability, this is the same as
the deterministic convergence
\[
\sqrt{k_n } ( m_p(k/n) - m_p) \to 0;
\] 
see the proof of Theorem \ref{thm:est-norm} for the precise version.
In case of the Hill estimator $(p=1)$ Cs\"org\H{o} and Viharos 
\cite{CsV95} obtained optimal conditions under which the 
random centralization $m_p(U_{k+1,n})$ in Theorem \ref{thm:asnorm}
can be replaced by the deterministic one $m_p(k/n)$.
For general residual estimator this was obtained in Theorem 4.2 
in \cite{Segers2}.
In Theorem 4.5 in \cite{Segers2} conditions were obtained which
assures that the random centering can be replaced by the limit $m_p$.
However, in Theorem 4.5 in \cite{Segers2} the slowly varying 
function $\ell$ belongs to the de Haan class $\Pi$, see the 
definitions below. Our assumptions are weaker.

\smallskip

We need second order conditions on the slowly varying function $\ell$.
First assume that
\begin{equation} \label{eq:ell-ass2}
\limsup_{v \downarrow 0} 
\sup_{u \in [0, 1]}
\frac{ |\ell(uv) - \ell(v)|}{a(v)} =: K_1 < \infty,
\end{equation}
where $a$ is a regularly varying function 
such that 
\begin{equation} \label{eq:a-ass}
\lim_{v \downarrow 0} \frac{a(v)}{\ell(v)} = 0. 
\end{equation}
In Proposition \ref{prop:m-conv}
we assume less stringent conditions on $\ell$, 
however in this case
it is easier to obtain the rate of convergence.

In the following two propositions we allow 
$p = p_v \to \infty$ at certain
rate, which we assume in the next section.

\begin{proposition} \label{prop:m-conv2}
On the slowly varying function assume \eqref{eq:ell-ass2}
and \eqref{eq:a-ass}. Further, assume
that 
\begin{equation} \label{eq:pav}
\lim_{v \downarrow 0} 
p_v \frac{a(v)}{\ell(v)} = 0. 
\end{equation}
Then there exists $v_0 > 0$
such that for all $v \in (0,v_0)$
\[
| m_{p_v} (v) - m_{p_v} | \leq 
2 K_1
\frac{a(v)}{\ell(v)}  \gamma^{p_v-1} \Gamma(p_v+1).
\]
\end{proposition}

Now we turn to more general conditions on the slowly varying
function $\ell$. We still need some kind of weak second 
order condition.
Assume that there is a regularly varying function $a$
for which \eqref{eq:a-ass} holds,
and a Borel set $B \subset [0,1]$ with positive measure, 
such that 
\begin{equation} \label{eq:ell-ass1}
\limsup_{v \downarrow 0} 
\frac{ |\ell(uv) - \ell(v)|}{a(v)} < \infty
\quad \text{for } \, u \in B.
\end{equation}
By Theorem 3.1.4 in Bingham et al.~\cite{BGT}
condition \eqref{eq:ell-ass1} implies that 
the limsup in \eqref{eq:ell-ass1} is finite
uniformly on any compact set of $(0,1]$.
However, in general, uniformity cannot be extended to $[0,1]$.
Put $a \vee b = \max\{ a, b\}$,
$a \wedge b = \min \{ a, b \}$.
Introduce the notation
\[ 
h(u) = u - 1- \log u, \quad u > 0,
\] 
and for $\beta \in (0,\infty]$
\begin{equation} \label{eq:nu-def}
\nu_\beta = \beta^{-1} h(2 \vee 2 \beta), \quad
\nu_\infty = 2. 
\end{equation}
Note that the weaker conditions on $\ell$ imply 
more restrictive conditions on $p$, when $p \to \infty$.

\begin{proposition} \label{prop:m-conv}
Assume  \eqref{eq:a-ass}, \eqref{eq:ell-ass1}, and 
\begin{equation} \label{eq:p-cond}
\beta := \liminf_{v \downarrow 0} 
\frac{ - \log \frac{a(v)}{\ell(v)}}{p_v} > 0,
\end{equation}
allowing $\beta = \infty$.
If $\nu_\beta > 1$ in \eqref{eq:nu-def}
then for any $\varepsilon > 0$ there exists $K > 0$ such
that for $v$ small enough
\[
| m_{p_v}(v) - m_{p_v} | \leq K
\frac{a(v)}{\ell(v)} (\gamma + \varepsilon)^{p_v} \, \Gamma(p_v+1).
\]
If $\nu_\beta \leq 1$ then for any $\varepsilon > 0$ 
there exists a $K> 0$ such that for $v$ small enough
\[
| m_{p_v}(v) - m_{p_v} | \leq 
K \left(\frac{a(v)}{\ell(v)} \right)^{\nu_\beta - \varepsilon}
(\gamma + \varepsilon)^{p_v} \, \Gamma(p_v+1).
\]
\end{proposition}

Note that if $p>0$ is fixed then $\beta = \infty$ and we obtain the 
same bound as in Proposition \ref{prop:m-conv2}.

We emphasize that we do not need exact second-order asymptotics
for $\ell$, only bounds. In particular, if $\ell$ belongs to
the de Haan class $\Pi$ (defined at 0) then the conditions
\eqref{eq:ell-ass1} and \eqref{eq:a-ass} holds; see Appendix B
in de Haan and Ferreira \cite{deHaan}, or Chapter 3 
in Bingham et al.~\cite{BGT}. Therefore, even 
in the special case $p=1$, i.e.~for the Hill estimator,
our next result is a generalization of Theorem 3.1 in \cite{Dekkers}.
The conditions in Theorem 4.5 in \cite{Segers2} are also
more restrictive.

\begin{theorem} \label{thm:est-norm}
Assume that \eqref{eq:a-ass} and \eqref{eq:ell-ass1} hold for $\ell$,
and $k_n$ is such that $k_n \to \infty$, $k_n / n \to 0$, and 
\[
\sqrt{k_n} \frac{a(k_n / n)}{\ell(k_n / n)} \to 0. 
\]
Then as $n \to \infty$
\[
\frac{1}{\sqrt{k_n}} \sum_{i=1}^{k_n}
\left[ 
\left( \log \frac{Q(1-U_{i,n})}{Q(1- U_{k_n+1,n})} \right)^p 
- \gamma^p \Gamma(p+1) \right] 
\stackrel{\mathcal{D}}{\longrightarrow} N(0, \sigma_p^2),
\]
and
\[
p \sqrt{k_n} \left( \widehat \gamma(n) - \gamma \right)
\stackrel{\mathcal{D}}{\longrightarrow} N(0, \widetilde \sigma_p^2),
\]
with $\sigma_p^2 = \gamma^{2p} (\Gamma(2p+1) - \Gamma(p+1)^2)$,
and $\widetilde \sigma_p^2 = 
\gamma^{2 \left({1}/{p} -1\right)} \sigma_p^2$.
\end{theorem}

\begin{proof}
The theorem is an immediate consequence of Theorem \ref{thm:asnorm}
and Proposition \ref{prop:m-conv}. Indeed, by Proposition \ref{prop:m-conv}
\[
\sqrt{k} \, | m_p(U_{k+1,n}) - m_p | \leq c 
\sqrt{k} \frac{a(U_{k+1,n})}{\ell(U_{k+1,n})} =
\sqrt{k} \frac{a(k/n)}{\ell(k/n)} 
\frac{a(U_{k+1,n})}{a(k/n)} \frac{\ell(k/n)}{\ell(U_{k+1,n})}.
\]
By the assumption $\sqrt{k} a(k/n) / \ell(k/n) \to 0$, while the 
last two factors tends to 1, since $a$ and $\ell$ are 
regularly  varying and $U_{k+1,n} \sim k/n$.

The central limit theorem for $\widehat \gamma(n)$ follows from
the previous result using the delta method, see
Agresti \cite[Section 14.1]{Agresti}.
\end{proof}

\section{Asymptotics for large $p$} \label{sect:infty}

In this section we assume that $p$ tends to infinity 
at a certain rate.
First we determine the asymptotic behavior of the moments
as $p \to \infty$.

\begin{lemma} \label{lemma:m-bound}
For any $\varepsilon > 0$ there is a $v_0 > 0$ and $p_0 > 0$ 
such that for $v \in (0,v_0)$, $p > p_0$
\[ 
(\gamma - \varepsilon)^p \, \Gamma(p+1) \leq 
m_p(v) \leq (\gamma + \varepsilon)^p \, \Gamma(p+1).
\] 
\end{lemma}

\begin{proof}
First note that if $X$ is a nonnegative random variable for which
$\p ( X > x) > 0$ for any $x$ then for any $K > 0$
\[
\E X^p \sim \E X^p I(X > K) \quad \text{as } p \to \infty.
\]
This implies that for any $\varepsilon > 0$ and $a > 0$ there exist 
$p_0 = p_0(\varepsilon, a)$ such that for $p > p_0$
\begin{equation} \label{eq:p-bound}
(1 - \varepsilon)^p  \, \E (X + a)^p \leq  \E X^p \leq 
(1 + \varepsilon)^p  \, \E (X - a)^p . 
\end{equation}
Using the Potter bounds (see \eqref{eq:potter-2}) 
and \eqref{eq:p-bound}, for 
any $A > 1$ and $\varepsilon > 0$ there exists $v_0 > 0$, and 
$p_0 > 0$ such that for $v \in (0,v_0)$, $p > p_0$
\[
\begin{split}
m_p(v) & = \E \left( 
\log \left( U^{-\gamma} \frac{\ell(Uv)}{\ell(v)} \right)
\right)^p \\
& \leq \E \left( \log 
\left( U^{-(\gamma+\varepsilon)} A \right) \right)^p \\
& \leq (\gamma + \varepsilon)^p \E \left( \log U^{-1} + 
\frac{\log A}{\gamma + \varepsilon} \right)^p \\
& \leq ((1+ \varepsilon)(\gamma + \varepsilon))^p \Gamma(p+1).
\end{split}
\]
Together with an analogous lower bound, the statement follows.
\end{proof}

Recall \eqref{eq:defY}. 
Let $Y(v), Y_1(v), Y_2(v), \ldots$ be iid random variables, and put 
\[ 
Z_n(p, v) = \sum_{i=1}^n Y_i(v)^p. 
\] 
The following results are analogous to Theorems 2.1
and 2.2 by Bogachev \cite{Bogachev}.
The main difficulty in our setup is the additional
parameter $v$, in which we need some kind of uniformity.
For the sequence $p = p_n$ let
\begin{equation} \label{eq:alpha}
\liminf_{n \to \infty} \frac{\log n}{p_n} = \alpha \geq 0. 
\end{equation}
Note that $\alpha > 0$ in \eqref{eq:alpha}
means that $p_n$ increases at most logarithmically. To obtain
a weak law of a large numbers we need that $\alpha > 1$.

\begin{proposition} \label{prop:p-lln}
If $\alpha > 1$ then there exists $v_0 > 0$ such that 
uniformly for $v \in (0, v_0)$ as $p_n \to \infty$
\[
\frac{Z_n(p_n,v) - n m_{p_n}(v)}{n m_{p_n}(v)} 
\stackrel{\p}{\longrightarrow} 0,
\]
that is for any $\varepsilon > 0$
\[
\lim_{n \to \infty} \sup_{v \in [0,v_0]}
\p \left( 
| Z_n(p_n,v) - n m_{p_n}(v) | \geq \varepsilon {n m_{p_n}(v)}
\right) = 0. 
\]
\end{proposition}

For the central limit theorem we need further restriction
on $p_n$. In the iid case treated by Bogachev the condition
is sharp in the sense that for $\alpha \in (0,2)$
non-Gaussian stable limit theorem holds, 
see \cite[Theorem 2.4]{Bogachev}.

\begin{proposition} \label{prop:p-clt}
If $\alpha > 2$ then uniformly on $[0,v_0]$ for some $v_0$ small enough
\[
\frac{Z_n(p_n,v) - n m_{p_n}(v)}{\sqrt{n} \sigma_{p_n}(v)}
\stackrel{\mathcal{D}}{\longrightarrow}
N(0,1),
\]
that is for any $x \in \R$
\[
\lim_{n \to \infty}
\sup_{v \in [0,v_0]}
\left| 
\p \left( \frac{Z_n({p_n},v) - n m_{p_n}(v)}{\sqrt{n} 
\sigma_{p_n}(v)} \leq x \right)
- \Phi(x) \right| =0,
\]
where $\Phi$ is the standard normal distribution function.
\end{proposition}

As a consequence we obtain the following.

\begin{theorem} \label{thm:p-asy}
Assume that $k_n \to \infty$, $k_n / n \to 0$, and 
$p_n \to \infty$. Let denote
\begin{equation} \label{eq:def-alpha}
\alpha = \liminf_{n \to \infty} \frac{\log k_n}{p_n}.
\end{equation}
If $\alpha > 1$ then
\[
\frac{1}{{k_n} m_{p_n}(U_{k_n+1,n})} 
\sum_{i=1}^{k_n} 
\left( \log \frac{Q(1-U_{i,n})}{Q(1- U_{k_n+1,n})} \right)^{p_n} 
\stackrel{\p}{\longrightarrow} 1.
\]
Furthermore, for $\alpha > 2$
\[
\frac{1}{\sqrt{k_n} \sigma_{p_n}(U_{k+1,n})} \sum_{i=1}^{k_n}
\left[ 
\left( \log \frac{Q(1-U_{i,n})}{Q(1- U_{k_n+1,n})} \right)^{p_n} 
- m_{p_n}(U_{k+1,n}) \right] 
\stackrel{\mathcal{D}}{\longrightarrow} N(0, 1).
\]
\end{theorem}

Note that both the centering and the norming is random. To change 
to deterministic values $m_{p_n}$ and $\sigma_{p_n}$ further 
assumptions are needed. Recall $\alpha$ in \eqref{eq:def-alpha}.

\begin{theorem} \label{thm:p-asy-det}
Assume that for the slowly varying function $\ell$,
\eqref{eq:ell-ass2} and \eqref{eq:a-ass} hold.
Furthermore, $k_n \to \infty$, $k_n / n \to 0$, and 
$p_n \to \infty$ such that
\[
p_n \frac{a(k_n/n)}{\ell(k_n/n)} \to 0.
\]
If  $\alpha > 1$ then
\[
\frac{1}{{k_n} m_{p_n}} 
\sum_{i=1}^{k_n} 
\left( \log \frac{Q(1-U_{i,n})}{Q(1- U_{k_n+1,n})} \right)^{p_n} 
\stackrel{\p}{\longrightarrow} 1.
\]
If $\alpha > 2$ assume additionally
\[
\limsup_{n \to \infty} 
p_n^{-1} \log \left( \sqrt{k_n} \frac{a(k_n /n)}{\ell(k_n/n)} \right)
= \mu
< \log 2 .
\]
Then
\[
\frac{1}{\sqrt{k_n} \sigma_{p_n}} \sum_{i=1}^{k_n}
\left[ 
\left( \log \frac{Q(1-U_{i,n})}{Q(1- U_{k_n+1,n})} \right)^{p_n} 
- m_{p_n} \right] 
\stackrel{\mathcal{D}}{\longrightarrow} N(0, 1).
\]
\end{theorem}

\begin{proof}
First note that $U_{k+1,n} n/k \to 1$ in probability, and since $a$ and $\ell$
are regularly varying functions $U_{k+1,n}$ can be changed to $k/n$.

For the first result we have to show that 
$m_{p}(k/n) / m_{p} \to 1$. This follows from 
Proposition \ref{prop:m-conv2} as in the proof 
of Theorem \ref{thm:est-norm}.

For the central limit theorem, 
$\sigma_{p}(k/n) / \sigma_{p} \to 1$ follows again from
Proposition \ref{prop:m-conv2}, thus 
$\sigma_{p}(U_{k_n +1, n}) /\sigma_{p} \to 1$ also follows
as above.
To change the centering,
using again Proposition \ref{prop:m-conv2} and Lemma \ref{lemma:m-bound}
\begin{equation} \label{eq:mdiff-aux1}
\begin{split}
\frac{\sqrt{k}}{\sigma_{p_n}}
|m_{p}(k/n) - m_{p}| 
& = 
\frac{m_{p} \sqrt{k}}{\sigma_{p}}
\frac{|m_{p}(k/n) - m_{p}|}{m_{p}} \\
& \leq c \sqrt{k} \frac{(\gamma + \varepsilon)^p}
{(\gamma - \varepsilon)^p}
\frac{\Gamma(p+1)}{\sqrt{\Gamma(2p+1)}} \frac{a(k/n)}{\ell(k/n)}.
\end{split}
\end{equation}
Taking logarithm and dividing by $p$ and using the Stirling formula
\[
\begin{split}
& \limsup_{p \to \infty} p^{-1} 
\log \left[ \sqrt{k} 
\frac{\Gamma(p+1)}{\sqrt{\Gamma(2p+1)}} \frac{a(k/n)}{\ell(k/n)} 
\right] 
\leq  - \log 2 + \mu < 0. 
\end{split}
\]
Since $\varepsilon > 0$ in \eqref{eq:mdiff-aux1} is as small as we 
wish, the result follows.
\end{proof}

Similarly, it is possible to obtain law of large numbers and central
limit theorem under the conditions of Proposition \ref{prop:m-conv}.
We do not go into further details.

Next we translate the previous result for our estimator.

\begin{theorem} \label{thm:p-gamma-est}
Assume that $k_n \to \infty$, $k_n / n \to 0$, and 
$p_n = \alpha^{-1} \log k_n$. If $\alpha > 1$ then
\[
\left( \frac{S_n(p_n)}{\Gamma(p_n+1)} \right)^{1/p_n}
\stackrel{\p}{\longrightarrow} \gamma.
\]
If $\alpha > 2$ then
\[
\frac{\sqrt{k_n} m_{p_n}(U_{k+1,n})}{\sigma_{p_n}(U_{k+1,n})}
p_n \left[ 
\left( \frac{S_n(p_n)}{m_{p_n}(U_{k+1,n})} \right)^{1/p_n} -1
\right]
\stackrel{\mathcal{D}}{\longrightarrow} N(0,1).
\]
Furthermore, under the conditions of Theorem \ref{thm:p-asy-det},
deterministic centering and norming works, i.e.
\begin{equation} \label{eq:est-detCLT}
\frac{\sqrt{k_n} m_{p_n}}{\sigma_{p_n}}
p_n \left[ 
\left( \frac{S_n(p_n)}{m_{p_n}} \right)^{1/p_n} -1
\right]
\stackrel{\mathcal{D}}{\longrightarrow} N(0,1).
\end{equation}
\end{theorem}

\begin{proof}
The first statement is an immediate consequence of 
Lemma \ref{lemma:m-bound} and Theorem \ref{thm:p-asy}. 

The second statement follows from Lemma 9.1 in
\cite{Bogachev} and Theorem \ref{thm:p-asy}.
To apply Lemma 9.1 in \cite{Bogachev} we only need 
to show that 
\[
\frac{\sqrt{k_n} m_{p_n}(U_{k+1,n})}
{\sigma_{p_n}(U_{k+1,n})} \to \infty.
\]
This follows 
easily from Lemma \ref{lemma:m-bound} as
\[
\liminf_{n \to \infty} p_n^{-1} \log 
\frac{\sqrt{k_n} m_{p_n}(U_{k+1,n})}{\sigma_{p_n}(U_{k+1,n})}
\geq \frac{\alpha}{2} - \log 2- \log (1 + \varepsilon) > 0.
\]
\end{proof}

\begin{example}
Assume that the slowly varying function $\ell$ in \eqref{eq:quant} has the form
\[ 
\ell(u) = c + O(u^{\delta}) \quad \text{with } \, c > 0, \delta > 0.
\] 
The asymptotic normality of the Hill estimator was proved 
for this subclass by Hall \cite{Hall}.
Conditions \eqref{eq:ell-ass2} and \eqref{eq:a-ass} are satisfied with
$a(u) = u^\delta$. By Proposition \ref{prop:m-conv2}
\[
| m_{p_n}(u) - m_{p_n} | \leq c \Gamma(p_n +1) u^{\delta}.
\]
If $p_n = \alpha^{-1} \, \log k_n$ with $\alpha > 2$ and
\begin{equation} \label{eq:Hall-cond}
\limsup_{n \to \infty} \frac{1}{p_n} \log 
\frac{k_n^{1/2 + \delta}}{n^{\delta}} < \log 2,
\end{equation}
then \eqref{eq:est-detCLT} holds. It is easy to see that 
\eqref{eq:Hall-cond} is satisfied if
$\log k_n = o(\log n)$.
\end{example}

\section{Simulation study} \label{sect:sim}

We provide simulation study for our estimators.
Note that for $p=1$ we obtain the usual Hill estimator.
In Theorem 5.1 Segers \cite{Segers2} proved the optimality
of the Hill estimator among residual estimators. 
We also see from Theorem \ref{thm:p-gamma-est}
that the asymptotic variance increases with $p$.
However, in practical situation higher $p$ values 
turns out to be useful as we show below.

In the simulations below $n=1000$ and we repeated the simulations
$5000$ times. In all the figures the mean and mean squared error (MSE)
are calculated for different values of $\gamma$ and $k_n$. 

In Table 
\ref{tab:par} we see that the Hill estimator is the best in the strict
Pareto model. In this case $Q(1-s) = s^{-\gamma}$. However,
in practice it is very unusual to encounter data which
fit to a nice distribution everywhere. It is more common that the large
values fit to a Pareto-type distribution, while the smaller values
behave as a light-tailed distribution.
Consider the quantile function 
\begin{equation} \label{eq:exp-par}
Q(1-s) =
\begin{cases}
s^{-\gamma}, & \text{if } s \leq 0.1, \\
\frac{10^{\gamma}}{\log 10} \log s^{-1}, & \text{if } s \geq 0.1,
\end{cases}
\end{equation}
which is a mixture of an exponential and a strict Pareto quantile.
The parameter of the exponential is chosen such that $Q$ is continuous.
Table \ref{tab:exp-par} contains the simulation results for $\gamma =1$.
In this simple model we already see the advantage of larger $p$ values.
Note that the Hill estimator is very sensitive to the change of $k_n$
for those values where the quantile function changes. Indeed, for 
$k_n \leq 100$ we basically have a sample from a strict Pareto distribution,
and for those values the Hill estimator is the best. For $k_n = 200$
we already see the exponential part of the sample, and the Hill estimator
changes drastically (from 0.98 to 0.76), while for $p = 5$ the change 
is not as large (from 0.92 to 0.88).

\begin{table} 
\begin{center}
\begin{tabular}{c||c|c|c} 
mean & $k=10$ & $k=50$ & $k=100$ \\
\hline \hline
$p=1$ & 0.9964 & 1.0001 & 1.0007 \\
\hline
$p=2$ & 0.9458 & 0.9878 & 0.9942 \\
\hline
$p=5$ & 0.7508 & 0.8946 & 0.9300 
\end{tabular}
\begin{tabular}{c||c|c|c}
MSE & $k=10$ & $k=50$ & $k=100$ \\
\hline \hline
$p=1$ & 0.1022 & 0.0194 & 0.0100 \\
\hline
$p=2$ & 0.1086 & 0.0229 & 0.0121 \\
\hline
$p=5$ & 0.1531 & 0.0512 & 0.0343 
\end{tabular}
\caption{
Mean and MSE in the strict Pareto model with $\gamma = 1$.}
 \label{tab:par}
\end{center}
\end{table}

\begin{table} 
\begin{center}
\begin{tabular}{l||c|c|c|c|c}
mean & $k=5$ & $k=10$ & $k=20$ & $k=100$ & $k=200$ \\
\hline \hline
$p=1$ & 1.0039 & 0.9968 & 1.0021 & 0.9790 & 0.7654 \\
\hline
$p=5$ &  0.6663 & 0.7469 & 0.8260 & 0.9238 & 0.8836\\
\hline
$p=10$ & 0.4387 &  0.5175 & 0.6009 & 0.7430 & 0.7480
\end{tabular}
\end{center}
\begin{center}
\begin{tabular}{l||c|c|c|c|c}
MSE & $k=5$ & $k=10$ & $k=20$ & $k=100$ & $k=200$ \\
\hline \hline
$p=1$ & 0.1981 & 0.1039 & 0.0493 & 0.0112 & 0.0593 \\
\hline
$p=5$ &  0.2241 & 0.1529 & 0.0967 & 0.0348 & 0.0344 \\
\hline
$p=10$ & 0.3663 & 0.2799 & 0.2011 & 0.0947 & 0.0883
\end{tabular}
\caption{
Mean and MSE for a sample with quantile function 
\eqref{eq:exp-par} with $\gamma = 1$.} 
\label{tab:exp-par}
\end{center}
\end{table}

Next, we further add a nonconstant slowly varying function
to the quantile. A logarithmic factor in the tail of 
the random variable cannot be detected, but it makes 
significantly more difficult to determine the underlying
index of regular variation. We modify the construction
in \eqref{eq:exp-par} and consider the quantile function
\begin{equation} \label{eq:exp-par-log}
Q(1-s) = 
\begin{cases}
s^{-\gamma} (\log s^{-1})^3, & \text{if } s \leq 0.1, \\
10^\gamma (\log 10)^2  \log s^{-1}, & \text{if } s \geq 0.1.
\end{cases}
\end{equation}
Note again that the function is continuous.
We see from the simulation results in Table \ref{tab:exp-par-log}
that in this setup the estimators with larger $p$ values 
work much better than the Hill estimator. These estimators are 
not so sensitive for the change in the nature of the 
quantile function.

\begin{table}
\begin{center}
\begin{tabular}{l||c|c|c|c|c}
mean & $k=5$ & $k=10$ & $k=20$ & $k=100$ & $k=200$ \\
\hline \hline
$p=1$ & 1.5019 & 1.5516 & 1.6387 & 1.9031 & 1.2517 \\
\hline
$p=5$ &  0.9777 & 1.1242 & 1.2807 & 1.5962 & 1.4835\\
\hline
$p=10$ & 0.6427 & 0.7760 & 0.9250 & 1.2507 & 1.2297
\end{tabular}
\end{center}
\begin{center}
\begin{tabular}{l||c|c|c|c|c}
MSE & $k=5$ & $k=10$ & $k=20$ & $k=100$ & $k=200$ \\
\hline \hline
$p=1$ & 0.6599 & 0.5325 & 0.5250 & 0.8519 & 0.0781 \\
\hline
$p=5$ &   0.2145 & 0.1845 & 0.2033 & 0.4061 & 0.2712 \\
\hline
$p=10$ & 0.2247 & 0.1396 & 0.0843 & 0.1147 & 0.0978
\end{tabular}
\caption{
Mean and MSE for a sample with quantile function 
\eqref{eq:exp-par-log} with $\gamma = 1$.}
 \label{tab:exp-par-log}
\end{center}
\end{table}

We also apply the estimator with different $p$ values
to real data. We chose the data set of 
Danish fire insurance losses, which consists
of 2167 fire losses in millions of Danish Kroner.
The data set is included in the R package evir, and 
was analyzed in \cite{Resn} and in \cite[Example 6.2.9]{EKM}.
In Figure \ref{fig:danish} we plotted the estimate for
$1/\gamma$, i.e.~we plotted $1/\hat \gamma(n)$
against $k_n$, to obtain the
Hill plot in \cite{Resn} for $p = 1$. Resnick \cite{Resn} used
various techniques to obtain smoother plot. In our setting
larger $p$ values naturally produces smoother plots.

\begin{figure} \label{fig:danish}
\includegraphics[height=8cm]{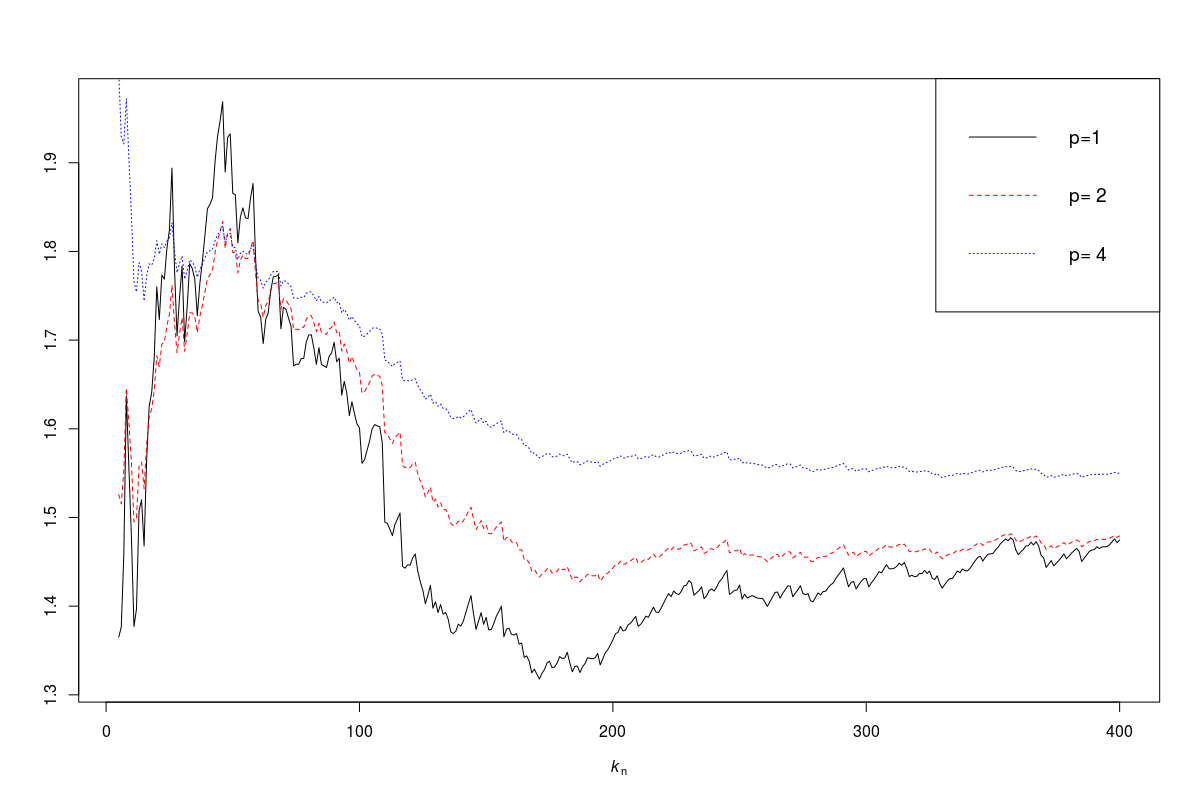}
\caption{Hill type plots 
of the estimator $\hat \gamma(n)^{-1}$
for the Danish fire insurance claim
with different $p$ values.}
\end{figure}

\section{Proofs} \label{sect:proofs}

\subsection{Strong consistency}

\begin{proof}[Proof of Lemma \ref{lemma:as}.]
Let $F_n$ denote the empirical distribution function of the 
sample $U_1$, $\ldots$, $U_n$. Then, integrating by parts, we have
\begin{equation} \label{eq:as-aux1}
\begin{split}
& \frac{1}{k} \sum_{i=1}^k 
\left( - \log \frac{U_{i,n}}{U_{k+1,n}} \right)^p 
 = \frac{n}{k} \int_{(0, U_{k,n}]} 
\left( - \log \frac{u}{U_{k+1,n}} \right)^p \dd F_n(u) \\
& =
\frac{n}{k}  \left[ 
F_n(U_{k,n}) \! \left( \! - 
\log \frac{U_{k,n}}{U_{k+1,n}} \right)^p \! + \!
\int_0^{U_{k,n}} \!\! F_n(u) \frac{p}{u} 
\left( - \log \frac{u}{U_{k+1,n}} \right)^{p-1} \! \dd  u
\right] \\
& = 
\left( - \log \frac{U_{k,n}}{U_{k+1,n}} \right)^p
+ p \frac{n}{k} \int_0^{U_{k,n}/U_{k+1,n}} F_n(U_{k+1,n}s)
( - \log s )^{p-1} \frac{1}{s} \dd s.
\end{split}
\end{equation}
Theorem 1 by Wellner \cite{Wellner78} implies that 
\begin{equation} \label{eq:well}
\frac{n}{k} U_{k,n} \to 1 \quad \text{a.s.~whenever } \,
k_n / \log \log n \to \infty.
\end{equation}
Thus, the first term in the right-hand side of 
\eqref{eq:as-aux1} tends to 0 a.s. For the second term
\[ 
\begin{split}
& \frac{n}{k} \int_0^{U_{k,n}/U_{k+1,n}} F_n(U_{k+1,n}s)
( - \log s )^{p-1} s^{-1} \dd s \\
&  = 
\frac{n}{k} U_{k+1,n} \int_0^{U_{k,n}/U_{k+1,n}}  ( - \log s )^{p-1}  \dd s \\
& \quad 
+ \frac{n}{k} \int_0^{U_{k,n}/U_{k+1,n}} ( F_n(U_{k+1,n}s) - U_{k+1,n} s)
( - \log s )^{p-1} s^{-1} \dd s \\
& =: I_n + II_n. 
\end{split}
\] 
Again by \eqref{eq:well}
\begin{equation} \label{eq:I}
I_n \to \int_0^1 (-\log s)^{p-1} \dd s = \Gamma(p)
\quad \text{a.s.}
\end{equation}
For the second term, choosing $\nu \in (0,1/2)$, we have
\begin{equation} \label{eq:II}
\begin{split}
II_n & \sim \int_0^1  \frac{F_n(U_{k+1,n} s) - U_{k+1,n} s }{U_{k+1,n} s} 
(-\log s)^{p-1} \dd s \\
& = \int_0^1 \frac{F_n(U_{k+1,n} s) - U_{k+1,n} s }{(U_{k+1,n} s)^{1/2 -\nu}} 
(-\log s)^{p-1} (U_{k+1,n} s)^{-1/2 -\nu} \dd s \\
& \leq \sup_{u \leq U_{k+1,n}} \frac{|F_n(u) - u|}{u^{1/2 -\nu}}
U_{k+1,n}^{-1/2 - \nu} 
\int_0^1  (-\log s)^{p-1} s^{-1/2 -\nu} \dd s \\
& \leq C 
\left( \frac{\log \log n}{k} \right)^{1/2} \, 
\left[ \left( \frac{n}{k} \right)^{\nu} 
\left( \frac{n}{\log \log n} \right)^{1/2}
\sup_{u \leq 2 k/n} \frac{|F_n(u) - u|}{u^{1/2 -\nu}}
\right],
\end{split}
\end{equation}
where $C > 0$ is a finite constant, not depending on $n, k_n$.
Using Theorem 1(ii) by Einmahl and Mason \cite{EinmahlMason88}
we see that the last term in \eqref{eq:II} is a.s.~bounded, 
if $k_n \geq (\log n)^{(1 - 2 \nu)/(2\nu)}$, which holds if 
$\nu$ is close enough to $1/2$. The first term in \eqref{eq:II}
tends to 0. 
From \eqref{eq:I}, \eqref{eq:II}, and \eqref{eq:as-aux1}
the statement follows.
\end{proof}

\begin{proof}[Proof of Theorem \ref{thm:strong-cons}.]
By the Potter bounds (\cite[Theorem 1.5.6]{BGT}), for any $A > 1$,
$\varepsilon > 0$ there exist $x_0 = x_0(A, \varepsilon)$ such that
\begin{equation} \label{eq:potter}
A^{-1} (y/x)^{-\varepsilon} \leq 
\frac{\ell(x)}{\ell(y)} \leq A (y / x)^\varepsilon \quad
\text{for any } \, 0 < x \leq y \leq x_0.
\end{equation}
Since $k / n \to 0$, equation \eqref{eq:well} implies 
$U_{k+1,n} \to 0$ a.s. Therefore, for $n$ large enough a.s.
\begin{equation} \label{eq:upper1}
\begin{split}
S_n(p) & = \frac{1}{k} \sum_{i=1}^k 
\left( \log \frac{U_{i,n}^{-\gamma} \ell(U_{i,n})}
{U_{k+1,n}^{-\gamma} \ell(U_{k+1,n})} 
\right)^p \\
& \leq \frac{1}{k} \sum_{i=1}^k 
\left( - (\gamma + \varepsilon) \log \frac{U_{i,n}}{U_{k+1,n}} + 
\log A \right)^p. 
\end{split}
\end{equation}

First let $p \leq 1$. Using the subadditivity
$(a + b)^p \leq a^p + b^p$, $a, b > 0$, 
by Lemma \ref{lemma:as} we obtain a.s.
\[
\begin{split}
\limsup_{n \to \infty} S_n(p) & \leq 
(\gamma + \varepsilon)^p
\limsup_{n \to \infty}
\frac{1}{k} \sum_{i=1}^k 
\left( - \log \frac{U_{i,n}}{U_{k+1,n}} \right)^p
+ (\log A)^p \\
& = (\gamma + \varepsilon)^p \Gamma(p +1) + (\log A)^p. 
\end{split}
\]
Letting $A \downarrow 1$ and $\varepsilon \downarrow 0$
we have a.s.
\[
\limsup_{n \to \infty} S_n(p)  \leq  \gamma^p \Gamma(p+1).
\]

Next, let $p > 1$. The convexity of the function $x^p$ implies 
that for any $\varepsilon' > 0$, for $a, b > 0$
\begin{equation} \label{eq:conv-ineq}
\begin{split}
( a + b)^p & \leq ( 1 + \varepsilon') a^p + 
\left( 1 - (1 + \varepsilon')^{-1/(p-1)} \right)^{-(p-1)} b^p \\
& =: ( 1 + \varepsilon') a^p + C_{\varepsilon'} b^p.
\end{split}
\end{equation}
Therefore, using Lemma \ref{lemma:as} and \eqref{eq:upper1},
we obtain a.s.
\[
\begin{split}
& \limsup_{n \to \infty} S_n(p) \\
& \leq 
(\gamma + \varepsilon)^p ( 1 + \varepsilon')
\limsup_{n \to \infty}
\frac{1}{k} \sum_{i=1}^k 
\left( - \log \frac{U_{i,n}}{U_{k+1,n}} \right)^p
+ C_{\varepsilon'} (\log A)^p \\
& =  (\gamma + \varepsilon)^p ( 1 + \varepsilon') \Gamma(p+1)
+ C_{\varepsilon'} (\log A)^p.
\end{split}
\]
As $A \downarrow 1$, $\varepsilon \downarrow 0$, 
$\varepsilon' \downarrow 0$, we have a.s.
\[
\limsup_{n \to \infty} S_n(p) \leq \gamma^p \Gamma(p+1). 
\]

For the lower bound choose $\varepsilon \in (0,\gamma)$. 
As in \eqref{eq:upper1}, by 
\eqref{eq:potter} for $n$ large 
enough a.s.
\[
S_n(p) \geq 
\frac{1}{k} \sum_{i=1}^k 
\left( - (\gamma - \varepsilon) \log \frac{U_{i,n}}{U_{k+1,n}} - 
\log A \right)_+^p,
\]
where $a_+ = \max \{a, 0\}$ stands for the positive part.
For $p \leq 1$ the subadditivity implies that
$(a-b)_+^p \geq a^p - b^p$ for $a,b>0$, while for $p > 1$
similarly as in \eqref{eq:conv-ineq}
\[
(a- b)_+^p \geq \frac{1}{1 + \varepsilon'} a^p -
\frac{C_{\varepsilon}}{1 + \varepsilon'} b^p.
\]
Using these inequalities, we obtain as above that a.s.
\[
 \liminf_{n \to \infty} S_n(p) \geq \gamma^p \Gamma(p+1),
\]
which completes the proof.
\end{proof}

\subsection{Asymptotic normality}

First we need two simple auxiliary lemmas.

\begin{lemma} \label{lemma:ab-ineq}
For $a \in (0, 1/2)$, $b \in (-1/2, 1/2)$,
and $a + b > 0$ we have
\[
| (a+b)^p - a^p | \leq 
\begin{cases}
p |b|, & p \geq 1, \\
2 |b| a^{p-1}, & p \leq 1.
\end{cases}
\]
\end{lemma}

\begin{proof}
Simply
$(a+b)^p - a^p = b p \xi^{p-1}$, with $\xi$ being between
$a$ and $a+b$. If $b > -a/2$ then $\xi \in [a/2, 1]$, thus
\[ 
\left| (a+b)^p - a^p \right| \leq |b| p 
\left( (a/2)^{p-1} \vee 1 \right).
\]
If $b < -a/2$ then $\xi \leq a$, thus 
$\xi^{p-1} \leq a^{p-1}$ for $p \geq 1$, and
\[ 
\left| (a+b)^p - a^p \right| \leq |b| p a^{p-1}. 
\] 
While if $b < -a/2$ and $p < 1$
\[
\begin{split}
|(a+b)^p - a^p| & = (a - |b| + |b|)^p - (a - |b|)^p \leq 
|b|^p \\
& = |b| |b|^{p-1} \leq |b| (a/2)^{p-1}.
\end{split}
\] 
\end{proof}

\begin{lemma} \label{lemma:Gamma-asy}
For $x \geq p > 0$ we have
\[
 \int_x^\infty e^{-y} y^p \dd y \leq  x^{p+1} e^{-x} (x-p)^{-1}.
\]
\end{lemma}

\begin{proof}
Simple calculation gives that 
\[
\begin{split}
\int_x^\infty e^{-y} y^p \dd y 
& = x^{p+1} e^{-x}
\int_1^\infty e^{-x (u-1) + p \log u} \dd u \\
& = x^{p+1} e^{-x} \int_1^\infty e^{-(x-p) (u-1) - p(u-1-\log u)} \dd u \\
& \leq x^{p+1} e^{-x} \int_1^\infty e^{-(x-p) (u-1)} \dd u \\
& = x^{p+1} e^{-x} (x-p)^{-1}.
\end{split}
\]
\end{proof}

\begin{proof}[Proof of Proposition \ref{prop:m-conv2}.]
To ease notation put
\begin{equation} \label{eq:def-eta}
\eta(u,v) = \left( - \gamma \log u + 
\log \frac{\ell(uv)}{\ell(v)} \right)^p - \left(-\gamma \log u \right)^p.
\end{equation}
We have by \eqref{eq:quant}
\begin{equation*}
\begin{split}
m_p(v) - m_p 
& =  \E \left[ \left( \log \frac{Q(1-Uv)}{Q(1-v)} \right)^p 
- \left( - \gamma \log U \right)^p \right] \\
& = \E \left[ \left( - \gamma \log U + \log \frac{\ell(Uv)}{\ell(v)}
\right)^p 
- \left( - \gamma \log U \right)^p
\right] \\
& = \int_0^1 \eta(u,v) \dd u 
 =: I_1(\delta) + I_2(\delta),
\end{split}
\end{equation*}
where $I_1$, $I_2$ are the integrals on
$(0, 1-\delta)$, $(1-\delta, 1)$, with $\delta \in (0,1/2)$.

First we deal with the integral on $(0,1-\delta)$.
By \eqref{eq:potter}, for any $\varepsilon > 0$, $A > 1$,
there is $v_0> 0$ such that for $v \leq v_0$, $u \in (0,1)$
\begin{equation} \label{eq:potter-2}
A^{-1} u^{\varepsilon} \leq  
\frac{\ell(uv)}{\ell(v)} \leq A u^{-\varepsilon},
\end{equation}
implying that uniformly on $u \in (0,1-\delta]$
\begin{equation} \label{eq:logell-1}
\frac{\log \frac{\ell(uv)}{\ell(v)}}{-\log u} \to 0
\quad \text{as } v \downarrow 0.
\end{equation}
Writing
\[
\frac{\ell(uv) - \ell(v)}{\ell(v)} =
\frac{a(v)}{\ell(v)} \frac{\ell(uv) - \ell(v)}{a(v)},
\]
we see that the first factor tends to 0 by \eqref{eq:a-ass}
and the second factor is bounded by \eqref{eq:ell-ass2}. 
Therefore, uniformly in $u \in [0, 1]$
\begin{equation} \label{eq:log-l}
\log \frac{\ell(uv)}{\ell(v)} \sim  
\frac{a(v)}{\ell(v)} \frac{\ell(uv) - \ell(v)}{a(v)}
\quad \text{as } v \downarrow 0.
\end{equation}
By \eqref{eq:logell-1} and \eqref{eq:log-l},
if \eqref{eq:pav} holds
then, uniformly on $u \in [0, 1-\delta]$,
\begin{equation} \label{eq:m-aux3}
\left( 1 + 
\frac{ \log \frac{\ell(uv)}{\ell(v)}}{- \gamma \log u} \right)^p -1 
\sim p (- \gamma \log u)^{-1} \, 
\frac{a(v)}{\ell(v)} \frac{\ell(uv) - \ell(v)}{a(v)}. 
\end{equation}
Thus, 
\begin{equation} \label{eq:I1-bound}
\begin{split}
I_1(\delta) \leq p \frac{a(v)}{\ell(v)} \,
\frac{3}{2} K_1 \gamma^{p-1} \int_{0}^{1-\delta}
\left( - \log u \right)^{p-1} \dd u.
\end{split}
\end{equation}

\smallskip
Next, we turn to $I_2$. Note that \eqref{eq:log-l}
holds, but \eqref{eq:logell-1} does not, because 
$\log u$ can be small. Choosing $\delta > 0$ small enough
we can achieve that $-\gamma \log (1-\delta) \in (0, 1/2)$ and
by \eqref{eq:log-l} also that $\log \ell(uv)/\ell(v) \in (-1/2,1/2)$
for $v$ small and $u \in [1-\delta, 1]$. Therefore, we can apply 
Lemma \ref{lemma:ab-ineq} with $a = -\gamma \log u$ and 
$b = \log (\ell(uv) / \ell(v))$ together with 
\eqref{eq:log-l} and \eqref{eq:ell-ass2}, and we obtain for $p \leq 1$
\[ 
\begin{split}
\left| \eta(u,v) \right| 
& \leq 2 \left| \log \frac{\ell(uv)}{\ell(v)} \right|
\, (-\gamma \log u)^{p-1} \\
& \leq \frac{a(v)}{\ell(v)} 2 K_1 (-\gamma \log u)^{p-1}.
\end{split}
\] 
While, for $p \geq 1$
\[ 
\begin{split}
\left| \eta(u,v) \right| 
 \leq p \left| \log \frac{\ell(uv)}{\ell(v)} \right| 
 \leq p \frac{a(v)}{\ell(v)} K_1.
\end{split}
\] 
Summarizing,
\begin{equation} \label{eq:I2-bound}
I_2(\delta) \leq
\begin{cases}
\frac{a(v)}{\ell(v)} 2 K_1 \gamma^{p-1} 
\int_{1-\delta}^1 (-\log u)^{p-1} \dd u, & p \leq 1, \\
p \frac{a(v)}{\ell(v)} K_1 \delta, & p \geq  1.
\end{cases}
\end{equation}
The bounds \eqref{eq:I1-bound} and  \eqref{eq:I2-bound}
imply the statement.
\end{proof}

\begin{proof}[Proof of Proposition \ref{prop:m-conv}.]
The difference compared to the previous proof is that 
\eqref{eq:ell-ass2} does not hold uniformly in $[0,1]$,
which implies that the integral on $[0,\delta]$ has to
be treated differently.

By Theorem 3.1.4 in \cite{BGT}
(translating the results from infinity to zero, 
by defining $\overline \ell(x) = \ell(x^{-1})$,
$\overline a(x) = a(x^{-1})$)
\[ 
\limsup_{v \downarrow 0} 
\sup_{u \in [\delta, 1]}
\frac{ |\ell(uv) - \ell(v)|}{a(v)} =: K_1(\delta) < \infty.
\] 
This implies that  the bound \eqref{eq:I2-bound}
on $[1-\delta,1]$ remains true and on $[\delta, 1-\delta]$
as in \eqref{eq:I1-bound} we have
\begin{equation} \label{eq:I1-bound-delta}
\int_\delta^{1-\delta} \eta(u,v) \dd u 
\leq p \frac{a(v)}{\ell(v)} \,
\frac{3}{2} K_1 \gamma^{p-1} \int_{\delta}^{1-\delta}
\left( - \log u \right)^{p-1} \dd u.
\end{equation}

Recall \eqref{eq:def-eta} and let
\begin{equation} \label{eq:J12}
J_1 = \int_0^{b(v)} \eta(u,v) \dd u, \quad
J_2 = \int_{b(v)}^{\delta} \eta(u,v) \dd u,
\end{equation}
where 
\begin{equation} \label{eq:def-b}
b(v) = \left( \frac{a(v)}{\ell(v)} \right)^2 \wedge e^{-2p}.
\end{equation}

By Theorem 3.1.4 in \cite{BGT} for any $\varepsilon > 0$
there is $v_0(\varepsilon) > 0$ and 
$K_2(\varepsilon) > 0$ such that
\begin{equation} \label{eq:ell-prop2}
\frac{|\ell(uv) - \ell(v)|}{a(v)} \leq 
K_2(\varepsilon) u^{-\varepsilon}
\quad \text{for all } u \leq 1, v \leq v_0(\varepsilon).
\end{equation}
By \eqref{eq:def-b} and \eqref{eq:p-cond} for $\varepsilon_1 > 0$ small
enough
\begin{equation} \label{eq:b-ass}
p \frac{a(v)}{\ell(v)} b(v)^{-\varepsilon_1} \to 0. 
\end{equation}
Using \eqref{eq:ell-prop2}, for $u \geq b(v)$
\[
\frac{|\ell(uv) - \ell(v)|}{\ell(v)} \leq K_2(\varepsilon_1) 
\frac{a(v)}{\ell(v)} u^{-\varepsilon_1} \leq 
K_2(\varepsilon_1)
\frac{a(v)}{\ell(v)} b(v)^{-\varepsilon_1}
\to 0,
\]
therefore
\[
\left| \log \frac{\ell(uv)}{\ell(v)} \right| \sim 
\frac{|\ell(uv) - \ell(v)|}{\ell(v)}
\leq K_2(\varepsilon_1) \frac{a(v)}{\ell(v)} u^{-\varepsilon_1}.
\]
By \eqref{eq:b-ass} for $u \in [b(v), \delta]$ 
the asymptotic equality in \eqref{eq:m-aux3} holds,
thus for $J_2$ in \eqref{eq:J12}
\begin{equation} \label{eq:I1-bound-1}
\begin{split}
J_{2} & \sim \int_{b(v)}^\delta (- \gamma \log u)^{p}
p (-\gamma \log u)^{-1} \frac{a(v)}{\ell(v)}
\frac{\ell(uv) - \ell(v)}{a(v)} \dd u \\
& \leq  p \frac{a(v)}{\ell(v)} K_2(\varepsilon_1)
\int_{b(v)}^{\delta}
(- \gamma \log u)^{p-1} u^{-\varepsilon_1} \, \dd u \\
& \leq 
p \frac{a(v)}{\ell(v)} K_2(\varepsilon_1)
(1-\varepsilon_1)^{-p} \gamma^{p-1} \Gamma(p),
\end{split}
\end{equation}
where at the last inequality we used that
\[
\begin{split}
\int_{0}^{1}
(- \log u)^{p-1} u^{-\varepsilon_1} \, \dd u & =
\int_0^\infty y^{p-1} e^{-(1-\varepsilon_1) y} \dd y \\
& = (1-\varepsilon_1)^{-p} \, \Gamma(p).
\end{split}
\]

On $(0,b(v))$ using \eqref{eq:potter-2}, $b(v) \to 0$,
Lemma \ref{lemma:Gamma-asy}, 
and that $-\log b(v) -p \geq (-\log b(v))/2$ we obtain
for $v$ small enough
\begin{equation} \label{eq:J1-1}
\begin{split}
J_{1} & \leq 2 \int_0^{b(v)} 
(-(\gamma + \varepsilon) \log u + \log A)^p \, \dd u \\
& \leq 2 (\gamma + 2 \varepsilon)^p 
\int_0^{b(v)} (-\log u)^p \, \dd u \\
& = 2 (\gamma + 2 \varepsilon)^p \int_{-\log b(v)}^\infty 
y^{p} e^{-y} \dd y \\
& \leq 2 (\gamma + 2 \varepsilon)^p
(-\log b(v))^{p+1} e^{\log b(v)} (-\log b(v) - p)^{-1} \\
&\leq 4 (\gamma + 2 \varepsilon)^p \,
(-\log b(v))^{p} \, b(v) .
\end{split}
\end{equation}

Note that for $\log x > p$
\[
\begin{split}
\frac{(\log x)^p}{x} \frac{e^p}{p^p} 
& = 
\exp \left\{ -p \left(
\frac{\log x}{p} - 1 - \log \frac{\log x}{p} \right) \right\} \\
& = \exp \left\{ - p h \left( \frac{\log x}{p} \right) \right\}.
\end{split}
\]
Thus with $x =b(v)^{-1}$
\[
\begin{split}
\left( \frac{e}{p} \right)^p (-\log b(v))^p \, b(v) 
& = \exp \left\{ - p h\left( 2 \vee \frac{- 2 \log (a(v) / \ell(v))}{p} \right)  
\right\} \\
& = \left( \frac{a(v)}{\ell(v)} \right)^{\frac{p}{-\log (a(v) / \ell(v)) }
h\left( 2 \vee \frac{- 2 \log (a(v) / \ell(v))}{p} \right)  }.
\end{split}
\]
Now the result follows from the monotonicity of $h$ and by the Stirling
formula. Indeed, continuing \eqref{eq:J1-1} for any $\varepsilon_2 > 0$
for $v$ small enough
\[
J_1 \leq  
\frac{4}{\sqrt{{p \pi}}} (\gamma + 2 \varepsilon)^p \,
\Gamma(p+1) 
\left( \frac{a(v)}{\ell(v)} \right)^{\nu_\beta - \varepsilon_2}.
\] 
Combining with \eqref{eq:I1-bound-1},  
\eqref{eq:I1-bound-delta}, and \eqref{eq:I2-bound} the result follows.
\end{proof}

\subsection{Asymptotics for large $p$}

\begin{proof}[Proof of Proposition \ref{prop:p-lln}.]
We follow the proof of Theorem 2.1 in \cite{Bogachev}.
Fix $\varepsilon > 0$, and let $r \in [1,2]$. 
Using the Markov inequality, the Marcinkiewicz--Zygmund
inequality (see e.g.~\cite[2.6.18]{Petrov}), and the subadditivity
we have
\begin{equation} \label{eq:Z-markov}
\begin{split}
& \p \left( \frac{|Z_n(p,v) - n m_p(v)|}{n m_p(v)} > \varepsilon \right) \\
& \leq (\varepsilon n  m_p(v))^{-r} 
\E |Z_n(p,v) - n m_p(v)|^r \\
& \leq c_r (\varepsilon n  m_p(v))^{-r} 
\E \left( \sum_{i=1}^n (Y_i(v)^p - m_p(v))^2 \right)^{r/2} \\
& \leq c_r (\varepsilon n  m_p(v))^{-r} n \E |Y(v)^p - m_p(v)|^r \\
& \leq c_r \varepsilon^{-r} n^{1-r} \frac{m_{rp}(v)}{m_p(v)^r}.
\end{split}
\end{equation}
By Lemma \ref{lemma:m-bound} for any $\varepsilon_1 > 0$ we can
choose $v_0 > 0$ and $p_0 > 0$ such that for $v \in (0, v_0)$ and 
$p > p_0$
\[
\frac{m_{rp}(v)}{m_p(v)^r} \leq 
\frac{(\gamma + \varepsilon_1)^{rp} \Gamma(rp+1)}
{(\gamma - \varepsilon_1)^{rp} \Gamma(p+1)^r}
\leq (1 + \varepsilon_2)^p \frac{\Gamma(rp+1)}{\Gamma(p+1)^r},
\]
with $\varepsilon_2 = 2 \varepsilon_2/(\gamma - \varepsilon_1)$.
Thus, by the Stirling formula
\begin{equation} \label{eq:bog-1}
\begin{split}
& \limsup_{p \to \infty} p^{-1} \log \frac{m_{rp}(v)}{ n^{r-1} m_p(v)^r} \\
& \leq \log (1 + \varepsilon_2) + r \log r - (r-1)
\liminf_{p \to \infty} \frac{\log n}{p} \\
& \leq \log ( 1 + \varepsilon_2) + r \log r - (r-1) \alpha.
\end{split}
\end{equation}
As $\alpha > 1$ we can choose $r \in [1,2]$ such that
$r \log r - (r-1) \alpha < 0$. Then choosing $\varepsilon_1$
small enough we see that the right-hand side in \eqref{eq:bog-1}
is negative, implying that 
the right-hand side in \eqref{eq:Z-markov} tends to 0.
\end{proof}

\begin{proof}[Proof of Proposition \ref{prop:p-clt}.]
By Lyapunov's theorem (see e.g.~Theorem 27.3 in 
Billingsley \cite{Bill})
it is enough to show that for some $\delta > 0$ uniformly in $v$
\[
\frac{n}{(\sqrt{n} \sigma_p(v))^{2+\delta}}
\E | Y(v)^p - m_p(v)|^{2 + \delta} \to 0
\]
as $n \to \infty$.
By Lemma \ref{lemma:m-bound} $\sigma_p(v) \sim \sqrt{m_{2p}(v)}$ as 
$p \to \infty$. Thus we have to show that 
\[
\frac{m_{p(2+\delta)}(v)}{n^{\delta/2} m_{2p}(v)^{1+\delta/2}} \to 0.
\]
As in the proof of Proposition \ref{prop:p-lln}
\[
\begin{split}
& \limsup_{p \to \infty} p^{-1} 
\log 
\frac{m_{p(2+\delta)}(v)}{n^{\delta/2} m_{2p}(v)^{1+\delta/2}} \\
& \leq - \frac{\delta}{2} \alpha + \log (1 + \varepsilon) 
+ (2 + \delta) \log (1 + \delta/2).
\end{split}
\]
We have to choose $\delta > 0$ such that
\[
\frac{2}{\delta} (2 + \delta) 
\log \left( 1 + \frac{\delta}{2} \right) < \alpha.
\]
This is possible for $\alpha > 2$.
\end{proof}

\bigskip

\noindent
\textbf{Acknowledgements.}
PK's research was partially supported 
by the J\'{a}nos Bolyai Research Scholarship of the 
Hungarian Academy of Sciences, 
 by the NKFIH grant FK124141, 
by the Ministry of Human Capacities, 
Hungary grant 20391-3/2018/FEKUSTRAT
and 
by the EU-funded Hungarian grant
EFOP-3.6.1-16-2016-00008.
LV's research was partially supported by the
Ministry of Human Capacities, Hungary grant 
TUDFO/47138-1/2019-ITM.

\end{document}